\documentclass{amsart}
\usepackage{amsmath,amsthm,amssymb,xypic,array}

\theoremstyle{plain}  

\newtheorem{theorem}{Theorem}[section]
\newtheorem{theoremn}{Theorem}
\newtheorem{proposition}[theorem]{Proposition}
\newtheorem{lemma}[theorem]{Lemma}

\theoremstyle{definition}

\newtheorem{definition}[theorem]{Definition}

\newtheorem{convention}[theorem]{Convention}

\theoremstyle{remark}

\newtheorem{claim}{Claim}

\newtheorem{case}{Case}

\newtheorem{remark}[theorem]{Remark}

\DeclareMathOperator{\Bl}{Bl}

\DeclareMathOperator{\Hilb}{Hilb}

\DeclareMathOperator{\cod}{cod}

\DeclareMathOperator{\Proj}{Proj}

\DeclareMathOperator{\Sing}{Sing}

\DeclareMathOperator{\Sec}{Sec}

\newcommand{\QED}{\ifhmode\unskip\nobreak\fi\quad {\rm Q.E.D.}} 

\newcommand\bin[2]{{#1\choose #2}}

\newcommand\iso{\cong}

\newcommand{\f}{\varphi}

\newcommand{\G}{\mathbb{G}}
\newcommand{\Gr}{{G}}

\renewcommand{\P}{\mathbb{P}}

\newcommand{\T}{\mathbb{T}}

\newcommand{\rat}{\dasharrow}
\renewcommand{\sec}{\mathbb{S}ec}

\newcommand{\Vg}[1]{VSP_G(F,#1)}
\newcommand{\Vh}[1]{VSP(F,#1)}

\begin{document}

\title{Birational aspects of the
  geometry of \\Varieties of Sum of Powers}

\author[Alex Massarenti]{Alex Massarenti}
\address{\sc Alex Massarenti\\
SISSA\\
via Bonomea 265\\
34136 Trieste\\ Italy}
\email{alex.massarenti@sissa.it}
\author[Massimiliano Mella]{Massimiliano Mella}
\address{\sc Massimiliano Mella\\ 
Universit\`a di Ferrara\\
Via Machiavelli 35\\
44100 Ferrara\\ Italy}
\email{mll@unife.it}
\date{October 2010}
\thanks{The second named author is
  partially supported by MIUR progetto PRIN
  ``Geometria delle variet\`a Algebriche''}
\maketitle

\begin{abstract} 
Varieties of Sums of Powers describe the
 additive decompositions of an
homogeneous polynomial into powers of
linear forms. Despite their long
history, going back to Sylvester and
Hilbert, few of them are known
for special degrees and number of
variables. In this paper we aim to understand a
general birational behaviour of \textit{VSP}, if
any. To do this we birationally embed
these varieties into Grassmannians and prove
the rationality, unirationality or
rational connectedness of many of those 
 in arbitrary degrees and
number of variables.
\end{abstract}


\section*{Introduction}
In 1770 Edward Waring stated that every
integer is a sum of at most 9 positive
cubes, later on Jacobi and others
considered the problem to find all the
decompositions of a given number into
a number of cubes,  \cite{Di}. Since
then many problems related to additive
decomposition are named after
Waring. The set up we are interested in
is that of homogeneous polynomials.
Let $F\in k[x_0,\ldots, x_n]_d$ be a general
homogeneous polynomial of degree
$d$. The additive decomposition we are
looking for is 
$$F=L_1^d+\ldots+L_h^d,$$
where $L_i\in k[x_0,\ldots, x_n]_1$ are
linear forms. The problem is very
classical, the first results are due to
\textit{Sylvester}, \cite{Sy} and then to
\textit{Hilbert}, \cite{Hi}, \textit{Richmond}, \cite{Ri},
\textit{Palatini}, \cite{Pa}, and many others. In the old
times the attention was essentially devoted to
study the cases in which the above
decomposition is unique. This gives a
canonical form to general homogeneous polynomials
of a particular degree and number of
variables. As widely expected the
canonical form very seldom exists, see \cite{Me1} \cite{Me2}.
In the remaining cases one should try to
understand the set of decompositions of a
given general polynomial. A
compactification of this is
usually called the \textit{Variety of Sums of Powers}
(\textit{VSP} for short), see Definition \ref{def:VSP} for
the precise statement.
The interest in these special varieties
increased greatly after \textit{Mukai}, \cite{Mu}, gave
a description of the Fano 3-fold
$V_{22}$ as  \textit{VSP} of quartic polynomials
in three variables. Since then different
authors exploited the area and
generalize  Mukai's techniques to
other polynomials, \cite{DK}, \cite{RS},
\cite{IR1}, \cite{IR2}, \cite{TZ}, see \cite{Do}
for a very nice survey.
All these works address special values of
degree and variables and give a
biregular description of special VSP's.
This is done studying the natural
compactification $VSP(F,h)$ of the additive
decompositions into the Hilbert scheme
of $(\P^n)^*$, see Section \ref{sec:pre} for
the details.  The known cases are not
many and, to the best of our
knowledge, this is the state of the art


\vspace{.2cm}\begin{center}
\begin{tabular}{|c|c|c|c|c|}
\hline
d & n & h & $VSP(F_{d},h)$ & Reference\\
\hline
2h-1 & 1 & h & 1 point & Sylvester \cite{Sy}\\
2 & 2 & 3 & quintic Fano threefold & Mukai \cite{Mu}\\
3 & 2 & 4 & $\mathbb{P}^{2}$ & Dolgachev
and Kanev \cite{DK}\\
4 & 2 & 6 & Fano 3-fold  $V_{22}$ & Mukai \cite{Mu}\\
5 & 2 & 7 & 1 point & Hilbert,\cite{Hi},
Richmond,\cite{Ri}\\
&&&& Palatini,\cite{Pa}\\
6 & 2 & 10 & K3 surface of genus 20 & Mukai \cite{Mu2}\\
7 & 2 & 12 & 5 points & Dixon and Stuart
\cite{Dx}\\
8 & 2 & 15 & 16 points & Mukai \cite{Mu2}\\
3 & 3 & 5 & 1 point & Sylvester's Pentahedral Theorem\cite{Sy}\\
3 & 4 & 8 &  $\mathcal{W}$ & Ranestad
and Schreier \cite{RS}\\
3 & 5 & 10 & $\mathcal{S}$ & Iliev and Ranestad \cite{IR1}\\
\hline
\end{tabular}\end{center}
\vspace{.2cm}

where $\mathcal{W}$ is a fivefold and is the variety of lines in the fivefold linear complete
intersection $\mathbb{P}^{10} \cap \mathbb{OG}(5,10) \subseteq \mathbb{P}^{15}$ of the ten-dimensional orthogonal Grassmanian $\mathbb{OG}(5,10)$, and $\mathcal{S}$ is a smooth symplectic fourfold obtained as a deformation of the Hilbert square of a polarized $K3$ surface of genus eight.

In this paper we aim to understand a
general birational behaviour of \textit{VSP}, if
any. To do this we prefer to adopt a different
compactification. This is probably less
efficient than the usual one to study
the biregular nature of these objects
but is very suitable for birational
purposes.

Let $F\in k[x_0,\ldots, x_n]_d$ be a
general homogeneous polynomial of degree
$d$ and $V = V_{d,n}\subset\P^N=\Proj(k[x_0,\ldots, x_n]_d)$ the Veronese
variety.
A general additive decomposition into
$h$ linear factors
$$F=\sum_1^h\lambda_iL_i^d$$
is associated to an $h$-secant linear
space of dimension $h-1$ to the Veronese $V\subset\P^N$.
In this way we can
embed a general additive decomposition
into either $\G(h-1,N)$, or $\G(h-2,N-1)$, and
consider the closure there.
This compactification is usually more
singular than the one into the Hilbert
scheme and meaningful only for $h<N-n$, see Remark \ref{rem:HvsG} for a
brief comparison. Nonetheless it allows us to use projective
techniques in a wider contest. In this
way we are able to give several
interesting result about the birational
nature of \textit{VSP}'s. 
\begin{theoremn} Assume that $F$ is a general
  quadratic polynomial in $n+1$
  variables. Then the irreducible
  components of $VSP(F,h)$ are
  unirational for any $h$
 and rational for $h=n+1$.
\label{thn:1}
\end{theoremn}
Theorem \ref{thn:1} cannot be extended
to higher degrees. Think for instance to
the mentioned examples of either \textit{Mukai}
or \textit{Iliev} and \textit{Ranestad}. On the other hand
we believe that Rational Connectedness
is the general pattern for this class of
varieties. In this direction our main result is the
rational connectedness of infinitely
many \textit{VSP} with arbitrarily high degree and number
of variable. Having in mind that $VSP(F,h)$ are
non empty only if $h\geq\bin{d+n}{n}/(n+1)$
we can prove the following.
\begin{theoremn} Assume that for some
 positive integer $k<n$ the number
  $\frac{\bin{d+n}{n}-1}{k+1}$ is an
  integer. Then the irreducible
  components of $\Vh{h}$ are Rationally
  Connected for $F\in k[x_0,\ldots,
    x_n]_d$ general and $h\geq \frac{\bin{n+d}{n}-1}{k+1}$.
\end{theoremn}

The common kernel of these theorems is
Theorem \ref{th:chain} where, under
suitable assumption,  we
connect $\Vh{h}$ with chains of $\Vh{h-1}$.  In this way
we reduce the rational
connectedness 
computations to special values of $h$. 
This new approach allows us also to recover
and reinterpret known classification
results, see section \ref{sec:old}, and
generalize a Theorem of
Sylvester, see section \ref{sec:rat}.


\section{Notation and Preliminaries}
\label{sec:pre} 
We work over the complex
field. We mainly follow notation and definitions of \cite{Do}.
Let $V$ be a vector space of dimension
$n+1$ and let $\mathbb{P}(V) =
\mathbb{P}^{n}$ the corresponding
projective space.
For any finite set of points $\{p_{1},...,p_{h}\} \subseteq \mathbb{P}^{n}$ we consider the linear space of homogeneous forms $F$ of degree $d$ on $\mathbb{P}^{n}$ such that $(F=0)$ contains the points $p_{1},...,p_{h}$, and we denote it by
\begin{center}
$L_{d}(p_{1},...,p_{h})=\{F \in k[x_{0},...,x_{n}]_{d} \: | \: p_{i} \in (F=0) \; \forall \; 1 \leq i \leq h\}$.
\end{center}

\begin{definition}
An unordered set of points $\{[L_{1}],...,[L_{h}]\} \subseteq \mathbb{P}V^{*}$ is a \textit{polar $h$-polyhedron} of $F \in k[x_{0},...,x_{n}]_{d}$ if
\begin{center}
$F=\lambda_{1}L^{d}_{1}+...+\lambda_{h}L^{d}_{h}$,
\end{center}
for some nonzero scalars $\lambda_{1},...,\lambda_{h} \in k$ and moreover the $L^{d}_{i}$ are linearly independent in $k[x_{0},...,x_{n}]_{d}$.
\end{definition}

We briefly introduce the concept of
Apolar form to a given homogeneous form
to state the connection between the set
of $h$-polyhedra of $F$ and the space of
Apolar forms of $F$.

Fix a system of coordinates
$\{x_{0},...,x_{n}\}$ on $V$ and the
dual coordinates
$\{\xi_{0},...,\xi_{n}\}$ on $V^{*}$.

Let $\phi = \phi(\xi_{0},...,\xi_{n})$
be a homogeneous polynomial of degree
$t$ on $V^{*}$. We consider the
differential operator 
\begin{center}
$D_{\phi} = \phi(\partial_{0},...,\partial_{n})$, with $\partial_{i} = \frac{\partial}{\partial x_{i}}$. 
\end{center}
This operator acts on $\phi$ substituting the variable $\xi_{i}$ with the partial derivative $\partial_{i} = \frac{\partial}{\partial x_{i}}$.
For any $F \in k[x_{0},...,x_{n}]_{d}$ we write 
\begin{center}
$<\phi,F> = D_{\phi}(F)$.
\end{center}
We call this pairing the \textit{apolarity pairing}.\\
In general $\phi$ is of the form $\phi(\xi_{0},...,\xi_{n}) = \sum_{i_{0}+...+i_{n}=t}\alpha_{i_{0},...,i_{n}}\xi_{0}^{i_{0}}...\xi_{n}^{i_{n}}$ and $F$ is of the form $F(x_{0},...,x_{n}) = \sum_{j_{0}+...+j_{n}=d}f_{i_{0},...,i_{n}}x^{j_{0}}_{0}...x^{j_{n}}_{n}$. Then
\begin{center}
$D_{\phi}(F)=(\sum_{i_{0}+...+i_{n}=t}\alpha_{i_{0},...,i_{n}}\partial_{0}^{i_{0}}...\partial_{n}^{i_{n}})(F)$.
\end{center}
We see that $F$ is derived
$i_{0}+...+i_{n} = t$ times. So we
obtain a homogeneous polynomial of
degree $d-t$ on $V$.

Fixed $F \in k[x_{0},...,x_{n}]_{d}$  we have the map
$$ap^{t}_{F} : k[\xi_{0},...,\xi_{n}]_{t} \rightarrow k[x_{0},...,x_{n}]_{d-t}, \: \phi \mapsto D_{\phi}(F).$$
The map $ap^{t}_{F}$ is linear and we can consider the subspace $Ker(ap^{t}_{F})$ of $k[\xi_{0},...,\xi_{n}]_{t}$.
\begin{definition}
A homogeneous form $\phi \in k[\xi_{0},...,\xi_{n}]_{t}$ is called \textit{Apolar} to a homogeneous form $F \in k[x_{0},...,x_{n}]_{d}$ if $D_{\phi}(F) = 0$, in other words if $\phi \in Ker(ap^{t}_{F})$. The vector subspace of $k[\xi_{0},...,\xi_{n}]_{t}$ of Apolar forms of degree $t$ to $F$ is denoted by $AP_{t}(F)$.
\end{definition}

\begin{lemma}[\cite{Do}]\label{ap}
The set $\mathcal{P} = \{[L_{1}],...,[L_{h}]\}$ is a polar $h$-polyhedron of $F$ if and only if
\begin{center}
$L_{d}([L_{1}],...,[L_{h}]) \subseteq AP_{d}(F)$,
\end{center}
and the inclusion is not true if we delete any $[L_{i}]$ from $\mathcal{P}$.
\end{lemma}

The set of all $h$-polyhedra of a
general polynomial $F \in
k[x_{0},...,x_{n}]_{d}$ is denoted by
$VSP(F,h)^{o}$. Via this construction it
is easy to embed $VSP(F,h)^{o}$ into
$\Hilb_h((\P^n)^*)$.
\begin{definition}
\label{def:VSP}  The closure 
$$VSP(F,h) := \overline{VSP(F,h)^{o}}
  \subseteq \Hilb_h((\P^n)^*)$$ 
is the \textit{Variety of Sums of Powers} of $F$.
 The points in $VSP(F,h)\setminus
 VSP(F,h)^{o}$ are called generalized
 polar polyhedra.
\end{definition}

 Using the smoothness of
 $\Hilb_h((\P^n)^*)$, when $n = 1,2$, one
 gets the following classical result,
 see for instance \cite{Do}.

\begin{proposition}\label{pro:hsm}
In the cases $n=1,2$ for a general polynomial $F \in k[x_{0},...,x_{n}]_{d}$ the variety $VSP(F,h)$ is either empty or a smooth variety of dimension
\begin{center}
$dim(VSP(F,h)) = h(n+1)-\binom{n+d}{d}$.
\end{center}
\end{proposition} 

It is important to notice that an additive decomposition of $F$ induces
an additive decomposition of its partial
derivatives.

\begin{remark}[Partial Derivatives]\label{pd}
Let $\{[L_{1}],...,[L_{h}]\}$ be a $h$-polar polyhedron for the homogeneous polynomial $F \in k[x_{0},...,x_{n}]_{d}$. We write
\begin{center}
$F = \lambda_{1}L^{d}_{1}+...+\lambda_{h}L^{d}_{h}$. 
\end{center}
The partial derivatives of $F$ are homogeneous polynomials of degree $d-1$ decomposed in $h$ linear factors
\begin{center}
$\frac{\partial F}{\partial x_{i}} = \lambda_{1}\alpha_{i_{1}}dL^{d-1}_{1}+...+\lambda_{h}\alpha_{i_{h}}dL^{d-1}_{h}$, for any $i = 0,...,n$.
\end{center}
Hence, as long as
$h<\bin{d-1+n}n$, $VSP(F,h)^{o} \subseteq
VSP(\frac{\partial F}{\partial
  x_{i}},h)^{o}$, and taking closures we have 
\begin{center}
$VSP(F,h)\subseteq VSP(\frac{\partial F}{\partial x_{i}},h)$.
\end{center}
The polynomial $F$ has $\binom{n+l}{l}$
partial derivatives of order
$l$. Clearly these derivatives are
homogeneous polynomials of degree $d-l$
decomposed in $h$-linear factors. Then,
when $h<\bin{d-l+n}n$, we have $VSP(F,h)\subseteq VSP(\frac{\partial^{l} F}{\partial x^{l_{0}}_{0},...,\partial x^{l_{n}}_{n}},h)$, where $l_{0}+...+l_{n} = l$.
\end{remark}

As remarked in the introduction we are
interested in a different
compactification of additive
decompositions.
Consider the span of the polar
polyhedron in the Veronese embedding.
We can associate to an $h$-polar
polyhedron of $F$ an $(h-1)$-plane
$h$-secant to the Veronese variety $V_{d,n}\subset\P^N$.
Hence, when $h<N-n+1$, we can embed 
$$VSP(F,h)^o\subset\G(h-1,N)$$
 as the
subvariety  formed
by the $(h-1)$-planes properly secant the
Veronese and containing $[F]$.
\begin{definition} Let 
$$\Vg{h}:=\overline{VSP(F,h)^o}\subset \G(h-1,N)$$
be the closure in the Grassmannian.
\end{definition}

\begin{remark}
\label{rem:HvsG} Note that  $\Vg{h}$ contains limits of
  $h$-secant planes. We expect, 
  in general, 
  that  there are no morphisms
  between the two compactifications but only rational
  maps. Indeed not all
  degree $h$ zero dimensional subschemes
  of the Veronese variety span a linear space of
  dimension $h-1$ and not all limits of
  $h$-secant planes cut a zero
  dimensional scheme. On the other hand
  both are clearly true when $n=1$ and in
  this case we have $\Vh{h}\iso\Vg{h}$.
\end{remark}

Let us recall, next, the main definitions and results concerning secant varieties.
Let $\Gr_{k-1}=\G(k-1,N)$ be the Grassmannian of $(k-1)$-linear spaces in $\P^N$.
Let $X\subset\P^N$ be an irreducible variety
$$\Gamma_k(X)\subset X\times\cdots\times X\times\Gr_{k-1},$$
 the closure of the graph of
$$\alpha:(X\times\cdots\times X)\setminus\Delta\to \Gr_{k-1},$$
taking $(x_1,\ldots,x_{k})$ to the  $[\langle
  x_1,\ldots,x_{k}\rangle]$, for $k$-tuple of distinct points.
Observe that $\Gamma_k(X)$ is irreducible of dimension $kn$. 
Let $\pi_2:\Gamma_k(X)\to\Gr_{k-1}$ be
the natural projection.
Denote by \label{pag:sec}
$$S_k(X):=\pi_2(\Gamma_k(X))\subset\Gr_{k-1}.$$
Again $S_k(X)$ is irreducible of dimension $kn$.
Finally let 
$$I_k=\{(x,\Lambda)| x\in \Lambda\}\subset\P^N\times\Gr_{k-1},$$
with natural projections $\pi^X_k$ and
$\psi^X_k$ onto the factors.
Observe that $\psi^X_k:I_k\to\Gr_{k-1}$ is a $\P^{k-1}$-bundle on $\Gr_{k-1}$.

\begin{definition} Let $X\subset\P^N$ be an irreducible variety. The {\it abstract $(k-1)$-Secant variety} is
$$\Sec_{k-1}(X):=(\psi^X_k)^{-1}(S_k(X))\subset I_k.$$ While the {\it $(k-1)$-Secant variety} is
$$\sec_{k-1}(X):=\pi^X_k(Sec_{k-1}(X))\subset\P^N.$$
It is immediate that $\Sec_{k-1}(X)$ is a $(kn+k-1)$-dimensional variety with a 
$\P^{k-1}$-bundle structure on $S_k(X)$. One says that $X$ is
\textit{$(k-1)$-defective}
if $$\dim\sec_{k-1}(X)<\min\{\dim\Sec_{k-1}(X),N\}$$
\end{definition}

\begin{remark} The definition of
  abstract and embedded $(k-1)$-Secants can
  be extended to the relative set up of
  varieties over a scheme $S$.
\label{rem:relsec}
\end{remark}
Let $F\in k[x_0,\ldots,x_n]_d$ be a
general polynomial. In the notation
introduced we have
$$\Vg{h}=\psi_h^{V_{d,n}}((\pi^{V_{d,n}}_{h})^{-1}([F]))\iso
(\pi^{V_{d,n}}_{h})^{-1}([F])
.$$
 In
what follows we more generally use fibers of
$\pi^X_k$-maps to study $VSP$
varieties. For this we slightly abuse
the language to state the following
\begin{definition} Let $X\subset\P^N$ be
  an irreducible variety, and $p\in\P^N$
  a general point. Then
$$VSP_G^X(h):=(\pi^X_{h})^{-1}(p)$$
It is important to stress that $VSP_G^X(h)$ is not
well defined as a variety. 
\end{definition}

\begin{remark}[Partial Derivatives II]
The partial derivatives Remark
\ref{pd} can be strengthened as follows.
Let $[F]\subset\P^N$ be a general
point. The partial
derivatives of $F$ span a linear space,
say $H_{\partial}$, in
the corresponding projective space
$\P^{N'}$.  Remark \ref{pd} tell us that
linear spaces
associated to polar polyhedra has to
contain $H_{\partial}$. In general
the opposite is not true but
for special values one could be lucky enough to get
equality, see for instance Theorems
\ref{DK}, and \ref{th:1}.
\label{rem:pdII}
\end{remark}

\section{A new Viewpoint on \textit{VSP}}
\label{sec:old} 
In this section we  want to give new
insight, and also test our ideas, on well known results about
$\Vh{h}$.
Let us start with a geometric proof of
Hilbert result on the uniqueness of
additive decomposition for quintic forms
in three variables.

\begin{theorem}[\cite{Hi}]
Let $F\in k[x_0,x_1,x_2]_5$ be a
general homogeneous polynomial. Then $\Vh{7}$ is a single point.
\end{theorem}
\begin{proof}
A computation, together with
\cite{AH} main result, shows that {\hbox{$\dim\Vh{7}=0$.}}
Assume that $F$ admits two different
decompositions, say
$\{[L_{1}],...,[L_{7}]\}$ and
$\{[l_{1}],...,[l_{7}]\}$. 
Consider the second
partial derivatives of $F$. Those are six
general homogeneous polynomials of
degree three. Let $H_{\partial}
\subseteq \mathbb{P}^{9}$ be the linear
space they generate. Then, by Remark
\ref{rem:pdII}, we have 
$$H_{L} :=
\langle[L_{1}^{3}],...,[L_{7}^{3}]\rangle\supset
H_\partial\subset
\langle[l_{1}^{3}],...,[l_{7}^{3}]\rangle=:H_l$$ 
The general choice of $F$ ensures that
both $H_{L}$ and $H_{l}$ intersect the Veronese
surface $V = V_{3,2} \subseteq \mathbb{P}^{9}$
 at 7 distinct points. \\
Let
$$\pi : \mathbb{P}^{9}\dashrightarrow
\mathbb{P}^{3}$$
be the projection from $H_\partial$, and
$\overline{V} = \pi(V)$. Then
$\overline{V}$ is a surface  of degree
$deg(V) = 9$ with two
points of multiplicity $7$ corresponding to $\pi(H_L)$ and
$\pi(H_l)$. This shows that the
7-dimensional linear space
$H:=\langle H_L,H_l\rangle$ intersects
$V$ along a curve, say $\Gamma$. The
construction of $\Gamma$ yields
$$\deg\Gamma\leq \#(H_L\cap V)=7.$$
On the other hand $\deg\Gamma=3j$
therefore we end up with the following
possibilities.

\begin{case}[$\deg\Gamma=3$] Then ${\Gamma}$ is a twisted cubic curve contained in $H$ and 
$$H_{l} \cdot {\Gamma} = H_{L} \cdot {\Gamma} = 3$$
We may assume that $H_{l} \cap {\Gamma} =
\{[l^{3}_{1}],[l^{3}_{2}],[l^{3}_{3}]\}$ and
$H_{L} \cap {\Gamma} =
\{[L^{3}_{1}],[L^{3}_{2}],[L^{3}_{3}]\}$. Let
$\Lambda$ be the pencil of hyperplanes  containing
$H$, and $\nu_3:\P^2\to V$ the
Veronese embedding. The linear system
$\nu_3^*(\Lambda_{|V})$ is a pencil of
conics and therefore 
$\#(\Bl\Lambda_{|V})\leq 4$. 

To conclude observe that
$\Bl\Lambda_{|V}\supset H\cap V$.
This force
$$\{[L^{3}_{4}],[L^{3}_{5}],[L^{3}_{6}],[L^{3}_{7}]\}
=
\{[l^{3}_{4}],[l^{3}_{5}],[l^{3}_{6}],[l^{3}_{7}]\},$$
 and consequently the contradiction $H_L=H_l$.
\end{case}
\begin{case}[$\deg\Gamma=6$] Then 
$$H_{l} \cdot \Gamma = H_{L}
\cdot \Gamma = 6$$
We may assume that
$\Gamma\supset\{[L_1^3],\ldots,[L_6^3]\}
\cup\{[l_1^3],\ldots,[l_6^3]\}$. 
Let
$\Lambda$ be the pencil of hyperplanes  containing
$H$. Let $\nu_3:\P^2\to V$ be the
Veronese embedding. The linear system $\nu_3^*(\Lambda_{|V})$ is
a pencil of lines and therefore
$\#(\Bl\Lambda_{|V})= 1$. 
This force
$$[L^{3}_{7}]
=
[l^{3}_{7}],$$
 and consequently the  contradiction $H_L=H_l$.
\end{case} 
\end{proof}

\textit{Sylvester Pentahedral Theorem} can be
proved similarly with a slightly more
involved argument. Giorgio Ottaviani
informed us of a very nice and neat
proof using apolarity, for this reason
we prefer to skip it. The final $\Vh{h}$
we are able to recover is  Dolgachev
and Kanev's result, \cite{DK} see also
\cite{RS}. This proof was actually the
starting point of our work.

\begin{theorem}[\cite{DK}] \label{DK}
Let $F\in k[x_0,x_1,x_2]_3$ be  a
general 
homogeneous polynomial. Then we have $\Vh{4}\iso\mathbb{P}^{2}$.
\end{theorem}
\begin{proof}
Let $\{[L_{1}],...,[L_{4}]\}$ be a $4$-polar polyhedron of $F$. By remark $\ref{pd}$ it is a $4$-polar polyhedron for the partial derivatives $\frac{\partial F}{\partial x_{0}}, \frac{\partial F}{\partial x_{1}}, \frac{\partial F}{\partial x_{2}}$ of $F$ also. Let
$$H_{\partial} = <\frac{\partial
  F}{\partial x_{0}}, \frac{\partial
  F}{\partial x_{1}}, \frac{\partial
  F}{\partial x_{2}}> \: \subseteq \:
\mathbb{P}(k[x_{0},x_{1},x_{2}]_{2})
\cong \mathbb{P}^{5}.$$
The $3$-planes, say $\Pi$, containing $H_{\partial}$
are parametrized by a plane. Moreover
the general choice of $F$ ensures that
$\Pi$ intersects $V = V_{2,2}$ along a zero
dimensional scheme of length 4. This
yields a bijective morphism
$$\phi : \mathbb{P}^{2} \rightarrow
  VSP(F,4), \: \Pi\mapsto \Pi\cap V.$$
The two varieties are smooth of the same
dimension by
Proposition \ref{pro:hsm} therefore $\phi$ is an isomorphism. 
\end{proof}

\section{Chains in $VSP(F,h)$}\label{sec:chain}
Let $F\in k[x_0,\ldots,x_n]_d$ be a
general 
homogeneous polynomial of degree
$d$. Consider a very general additive decomposition 
$$F=\sum_1^h\lambda_i L_i^d $$
Let $p\in VSP(F,h)^o$ the corresponding
point. For $p\in
VSP(F,h)^o$ general also the polynomial
$$F-\lambda_1L_1^d $$
is general and we can view a dense open of
$VSP(F-\lambda_1L_1^d,h-1)^o$ as a
subvariety of $VSP(F,h)^o$. More
generally we can consider a flag of
subsets
$$VSP(F,h)^o\supset
VSP(F-\lambda_1L_1^d,h-1)^o\supset\ldots\supset
VSP(F-\sum_1^r\lambda_iL_i^d,h-r)^0\ni
p $$
Passing to the closure in any of the
possible compactifications we can
cover any Variety of Sum of Powers
via $VSP$ with less addends. 
\begin{convention} When working with a
  general decomposition, say $\sum_1^h\lambda_i L_i^d$,we will always
  tacitly consider the irreducible
  component of $VSP(F,h)^o$ containing this
  general decomposition and keep
  denoting its compactifications 
$VSP(F,h)$, and
  $\Vg{h}$.\label{con:irr}
\end{convention}

Further note
that, when $\Vh{h-1}$ is not empty,
\begin{equation}
\label{eq:cod}
\cod_{VSP(F,h)^o}VSP(F-\lambda_1L_1^d,h-1)^o=n+1
\end{equation}
therefore as long as 
$$\dim VSP(F-\lambda_1L_1^d,h-1)^o\geq
n+1$$
we have a well defined intersection theory for these
subvarieties in any compactification.
Let $q\in
VSP(F-\lambda_1L_1^d,h-1)^o\subset
VSP(F,h)^o$ be a general point then $q$
represents a decomposition say
$$F=\lambda_1 L_1^{d}+\sum_2^h\mu_jG_j^d.$$
As long as $q\neq p$
we may assume that  
$\mu_2G_2\neq \lambda_iL_i$ for any
$i=1,\ldots h$. This means that 
$$p\not\in VSP(F-\mu_2G_2^d,h-1)^o.$$
Assume that $\dim
VSP(F-\xi_iM_i^d,h-1)\geq n+1$. Then
 by noetherianity we may assume that for
a  very general point $p_1$
in $VSP(F,h)^o$ there is $M_1$ such that 
$$p_1\in VSP(F-\xi_1M_1^d,h-1)^o $$
and
$$VSP(F-\xi_iM_i^d,h-1)\cap
VSP(F-\mu_2G_2^d,h-1)\neq \emptyset$$
That is we can connect two very general
points of $VSP(F,h)$ with a chain of
varieties of type
\hbox{$VSP_G^{V_{d,n}}(h-1)$.}
We make the above argument explicit in the
following Theorem.
\begin{theorem} Let $F\in
  k[x_0,\ldots,x_n]_d$ be a general
  polynomial of degree $d$. Assume that
  $h\geq\frac{\bin{n+d}d}{n+1}+2 $. Then
  two very general points of an
  irreducible component of $VSP(F,h)$
  are joined by a chain (of length
  at most three) of
  $VSP_G^{V_{d,n}}(h-1)$. Let $W_i$ be
  the 
  elements of this chain, and $q\in W_j\cap
  W_l$ a general point. Then we may
  assume that $q$ is a general point in 
  $\Vh{h}$, $W_j$, and $W_l$. 

Assume moreover that any irreducible
component of $VSP(F,h-1)$ is
Rationally Connected and $\dim
\Vh{h-1}\geq n$ then any irreducible
component of $\Vh{h}$ is
Rationally Connected.
\label{th:chain}
\end{theorem}
\begin{proof} We have
$$\dim VSP(F,h-1)=
  n(h-1)+h-2-\bin{n+d}d+1=(h-1)(n+1)-\bin{n+d}d.$$
Therefore, under our numerical
  assumption, this yields
\begin{equation}\dim
  VSP(F,h-1)-(n+1)=(n+1)(h-2)-\bin{n+d}d\geq
  0. 
\label{eq:chain}
\end{equation}
Let $p_1$ and $p_2$ be two very general
points in $VSP(F,h)$, with associated decompositions,
respectively,
$$\sum_1^h \lambda_iL_i^d \ {\rm and}\ \sum_1^h \mu_iG_i^d$$
Let $q\in VSP(F-\lambda_1L_1^d, h-1)$ be a
general point with
associated decomposition
$$\lambda_1L_1^d+\sum_2^h\xi_i B_i^d$$
Let $\nu:Z\to\Vh{h}$ be a resolution of
singularities. Since $p_1$ and $p_2$ are very
 general we may assume the following:
\begin{itemize}
\item[($\star$)]
$\nu^{-1}(VSP(F-\lambda_1L_1^d, h-1))$ and
$\nu^{-1}(VSP(F-\mu_1G_1^d, h-1))$ belong to the
same irreducible component of
$\Hilb(Z)$, and $\nu$ is an isomorphism
in a neighbourhood of $q$. 
\end{itemize}
Then by construction we have
$$q\in VSP(F-\lambda_1L_1^d, h-1)\cap
VSP(F-\xi_2B_2^d, h-1). $$
Hence by equations (\ref{eq:cod}), (\ref{eq:chain}),
and assumption ($\star$) we conclude
that 
$$VSP(F-\mu_1G_1^d, h-1)\cap
VSP(F-\xi_2B_2^d, h-1)\neq\emptyset.$$
Furthermore the general point of this
intersection is a general point of
$\Vh{h}$, $VSP(F-\mu_1G_1^d, h-1)$ and
$VSP(F-\xi_2B_2^d, h-1)$.

To conclude the rational connectedness
in case $\dim\Vh{h-1}=n$ consider the
variety
$V:=\overline{\bigcup_{\lambda}VSP(F-\lambda
  L_1^d, h-1)}$. Then $V$ has a natural map
onto $\P^1$ with Rationally Connected
fibers. Therefore, via the
main result of \cite{GHS}, we know that
$V$ is  a
Rationally Connected variety of
dimension $n+1$. Then we may argue  as
before with $V$ instead of $VSP(F-\lambda_1L_1^d, h-1)$.
\end{proof}

Theorem \ref{th:chain}
allows us to describe birational properties of $\Vh{h}$
starting from those of $VSP_G^{V_{d,n}}(h-1)$.

Our next task is to understand how to use $VSP_G(F',h)$,
with $F'\in k[x_0,\ldots,x_{n-k}]_d$ to
study $\Vg{h}$. It is difficult, at least
to us, to understand it from the
algebraic point of view. On the other
hand a neat geometric way is at hand.
Let $V_{d,n}\subset\P^N$ be the $d$-uple
Veronese embedding of $\P^n$. Let
$Y\subset V$ be a rational subvariety of
dimension $b$. Let us think of  $Y$ 
as the projection of 
$V_{\delta,b}\subset\P^{N'}$ for some $\delta$. 

\begin{theorem} Let $Y\subset\P^N$ be a
  projection of
  $V = V_{\delta,b}\subset\P^{N'}$, and
  $p\in\P^N$ a general point. Assume
  that $\sec_{h-1}(V)=\P^{N'}$ and
  $h<N-b$. Then there
  is an irreducible component $W\subset
  VSP^Y_G(h)$ containing general
  $h$-secant linear spaces, and a birational map
  $\f$ giving rise to the following
  diagram

\[
\xymatrix{
\hspace{-2cm}\Sec_{h-1}(V)\supset\tilde{W}\ar@{.>}^\f[d]\ar[rr]^{\pi_{h-1|\tilde{W}}}
&&\P^{N'-N} \\
\hspace{-2cm}\Sec_{h-1}(Y)\supset W&&}
\]
In particular if $VSP_G^V(h)$ is
Rationally Connected then $VSP_G^Y(h)$
has a Rationally connected irreducible
component of dimension
$N'-2N+(b+1)h-1$. 
\label{th:RC}
\end{theorem}
\begin{remark} The hypothesis of
  Theorem \ref{th:RC} are quite strong. In
  particular the assumption on
  $\Sec_{h-1}(V)$ confines its 
  application to finitely many cases. On
  the other hand we think it
  is important, at least conceptually,
  to have a geometric way to ``lower
  the number of variables''.
\end{remark}
\begin{proof}
Let $\pi:\P^{N'}\rat \P^N$ be
the projection, with $\pi(V)=Y$,
$p\in\P^N$ a general point, and $\Pi=\pi^{-1}(p)\iso\P^{N'-N}$.
 Since
$p\in\P^N$ is a general point we may
 assume that the general point of $\Pi$
 is general in $\P^{N'}$. Consider
 $\pi_{h-1}:\Sec_{h-1}(V)\to \P^{N'}$
 and let
 $\tilde{W}\subset\pi_{h-1}^{-1}(\Pi)$ be an
 irreducible component containing
 general $h$-secant linear spaces. 
The numerical assumption $h<N-b$ ensures
that:
\begin{itemize}
\item[i)] the general $h$-secant linear space
  to $V$ does
not intersect the center of projection,
\item[ii)] the general $h$-secant linear
  space to $Y$ intersect $Y$ in exactly $h$ points.
\end{itemize}

Item i) and ii) ensures that the general
$h$-secant linear space 
  to $V$ is mapped  to a general $h$-secant
linear space to $Y$. This gives the map
$\f$. Item ii) shows that $\f$ is  generically injective.
Note that the expected dimension of
$\sec_{h-1}(Y)$ is $(h+1)b-1-N$ while
$\dim \tilde{W}=(h+1)b-1-N'+(N'-N)$.
Therefore $Y$ is not defective and
$W:=\f(\tilde{W})$ is an irreducible component.

If $VSP_G^V(h)$ is
Rationally connected we may apply the
main result of \cite{GHS} to conclude
that $\tilde{W}$ is rationally connected
and henceforth $W$ is rationally connected.
\end{proof}

As we already noticed hypothesis of Theorem
\ref{th:RC} are seldom satisfied. The
following is a good tool to study Rational
Connectedness of $\Vg{h}$ in many more contests.

\begin{proposition}\label{pro:W}  
For any triple of integers
$(a,b,c)$, with $b<n$,
there is a Rationally Connected variety
$W^n_{a,c,b}\subset\Hilb(\P^n)$ with the following properties:
\begin{itemize}
\item[-] a general point in $W^n_{a,c,b}$
  represents  a rational
  subvariety of $\P^n$ of codimension $b$;
\item[-] for any  $Z\subset \P^n\setminus\{(x_0=\ldots=x_{n-b}=0)\}$ reduced zero dimensional
scheme of length $\leq
c$, there is a Rationally Connected subvariety
$W_{Z,b}\subset W^n_{a,c,b}$, of
  dimension at least $a$,
  whose general element $[Y]\in W_{Z,b}$
  represents a rational  subvariety of
  $\P^n$ of codimension $b$
  containing $Z$.
\end{itemize}
\end{proposition}
\begin{proof}
We prove the statement  by induction on
$b$.
Assume $b=1$, and
consider an equation of the form 
$$Y=(x_nA(x_0,\ldots,x_{n-1})_{d-1}+B(x_0,\ldots,x_{n-1})_d=0),$$
then, for $A$ and $B$ generic, $Y$ is a
rational hypersurface of degree $d$ with a unique
singular point of multiplicity $d-1$ at
the point $[0,\ldots,0,1]$. 


Fix $d>ac$ and let $W^n_{a,c,1}$ be the linear span of these
hypersurfaces. For any triple
$(a,1,c)$ and a subset
$Z\subset\P^n\setminus\{[0,\ldots,0,1]\}$ consider
$W_{Z,1}\subset W^n_{a,c,1}$ as the sublinear
system of hypersurfaces containing $Z$.

Assume, by induction, that
$W^n_{a,c,i-1}\subset\Hilb(\P^{n-1})$ exist for any $n$ and $c$.
Define, for $i\geq 2$,
$$\tilde{W}^n_{a,c,i}:= W^n_{a,c,1}\times
W^{n-1}_{a,c,i-1}\subset\Hilb(\P^n)\times\Hilb(\P^{n-1}).$$ 
Let $[X]$ be a general point in
$W^n_{a,c,1}$. By construction $X$ has a point of
multiplicity $d-1$ at the point
$[0,\ldots,0,1]\in \P^n$. Then the
projection
$\pi_{[0,\ldots,0,1]}:\P^n\rat\P^{n-1}$
restricts to a birational map $\f_X:X\rat
\P^{n-1}$. Hence we may
associate the general element
$([X],[Y])\in\{[X]\}\times
W^{n-1}_{a,c,i-1}$ to the
codimension $i$ subvariety
$\f_X^{-1}(Y)\subset\P^n$.
This, see for instance \cite[Proposition
I.6.6.1]{Ko}, yields a rational map 
$$\chi:\tilde{W}^n_{a,c,i}\rat\Hilb(\P^n).$$
Let $W^n_{a,c,i}:=\overline{\chi(\tilde{W}^n_{a,c,i})}\subset\Hilb(\P^n).$
For any $Z$ we may then define 
$$\tilde{W}_{Z,i}:=W_{Z,1}\times
W_{\pi_{[1,0,\ldots,0]}(Z),i-1},$$
and as above $W_{Z,i}=\overline{\chi(\tilde{W}_{Z,i})}$.
\end{proof}

\section{Rationality Results}\label{sec:rat} 

In this section we prove some
rationality result for $VSP$'s.
The first interesting case is that of
$\P^1$, namely polynomials in two
variables. This is probably 
known but we where not able to find an
appropriate reference.

\begin{theorem}
Let $h > 1$ be a fixed integer. For any integer $d$ such that 
\begin{center}
$h \leq d \leq 2h-1$,
\end{center}
we have $VSP(F,h)\iso \mathbb{P}^{2h-d-1}$.
\label{th:1}
\end{theorem}
\begin{proof}
We already noticed, see Remark \ref{rem:HvsG}, that in this case
$$\Vh{h}\iso\Vg{h}.$$ 
Let $F$ be a homogeneous polynomial of degree $d$ and let $\{[L_{1}],...,[L_{h}]\}$ be a $h$-polar polyhedron of $F$, then
\begin{center}
$F = \lambda_{1}L^{d}_{1}+...+\lambda_{h}L^{d}_{h}$.
\end{center}
We consider the partial derivatives of order $d-h > 0$ of $F$. This partial derivatives are 
\begin{center}
$\binom{d-h+1}{d-h} = d-h+1\leq h$ 
\end{center}
homogeneous polynomials of degree $h$.\\
Let $X$ be the rational normal curve of
degree $h$ in $\mathbb{P}^{h}$. The
partial derivatives span a $(d-h)$-plane
$H_{\partial}\subset\P^h$. The general
choice of $F$ ensures that
$H_\partial\cap X=\emptyset$. By  Remark
\ref{rem:pdII} the points
$[L^{d}_{1}],...,[L^{d}_{h}] \in X$ span
an hyperplane containing
$H_{\partial}$. 

The
hyperplanes of $\mathbb{P}^{h}$
containing $H_\partial$ are parametrized by
$\mathbb{P}^{2h-d-1}$ and any
hyperplane containing $H_\partial$
intersects $X$ in a zero dimensional scheme of
length $h$.
 This gives rise to an injective morphism
\begin{center}
$\phi : \mathbb{P}^{2h-d-1}\to VSP(F,h),
  \: \Pi\mapsto \Pi\cap X$.
\end{center}
The varieties $VSP(F,h)$ and
$\mathbb{P}^{2h-d-1}$ are both smooth by
Proposition \ref{pro:hsm}
and 
\begin{center}
$dim(VSP(F,h)) = 2h - \binom{d+1}{d} =
  2h - d - 1$.
\end{center}
Hence the injective morphism $\phi$ is an isomorphism. 
\end{proof}

The next rationality result is for
quadratic polynomials. It is well known
that two general quadrics can be
simultaneously diagonalized. Building on
this we can prove the following.

\begin{theorem}\label{th:rq}
Let $F\in k[x_0,\ldots,x_n]_2$ be a general homogeneous
polynomial of degree two. Then $VSP(F,n+1)$ is
rational. 
\label{th:2rat}
\end{theorem}
\begin{proof} 
Up to a projectivity of $\P^n$  we may
assume that  $F$ is given by
$$F = x^{2}_{0}+...+x^{2}_{n}.$$
Let $\Pi$ be a general $(N-n)$-plane in
$\mathbb{P}^N=\P(k[x_0,\ldots,x_n]_2)$, and
$[G]\in \Pi$ a general point.

The quadrics $F$ and $G$ are
general. Then we may assume that the
pencil they generate contains exactly
$n+1$ distinct singular quadric cones,
say  $C_{0},...,C_{n}$. Let  $v_{i} \in
\mathbb{P}^{n}$ the vertex of the cone
$C_{i}$ for $i=0,...,n$.  Via the
Veronese embedding
$\nu_{2}:\mathbb{P}^{n}\rightarrow\mathbb{P}^{N}$
we find $n+1$ points $\nu_{2}(v_{i})$ on
the Veronese variety
$V_{2,n}\subset\P^N$.

Let $A$ be the matrix of $G$. Then the
cones in the pencil $\lambda F-G$ are
determined by the values of $\lambda$
such that $det(\lambda I-A) = 0$. In
other words the cones $C_i$ correspond
to the eigenvalues of $A$ and the
singular points $v_i$ are given by the
eigenvectors of $A$.
In particular $v_i$'s are linearly
independent and in the basis
$\{v_{0},...,v_{n}\}$ the matrix $A$ is diagonal
$$\left(
\begin{array}{ccc}
\lambda_{0} & \cdots & 0\\
\vdots & \ddots & \vdots\\
0 & \cdots & \lambda_{n}\\
\end{array}
\right)$$

We may further assume that  $\{v_{0},...,v_{n}\}$ is
an orthonormal base. Therefore after the
projectivity induced by this change of
variables we have that $F$ is
still represented by the identity and $G$ is diagonal.

Any projectivity of $\P^n$ induces a
projectivity on $\P^N$ that stabilizes
$V\subset\P^N$. Hence after the needed
projectivities we have
$$\nu_{2}(v_{i}) = \nu_{2}([0,\ldots,0,1,0,\ldots,0]) =
[x^{2}_i]$$
Therefore the linear space
$<[x^{2}_{0}],...,[x^{2}_{n}]>$ contains
both $[F]$ and $[G]$. 
This construction gives a map
$$\psi:\Pi\rat VSP(F,n+1), [G]
\mapsto \{v_{0},...,v_{n}\}$$
The birationality of $\psi$ is immediate
once remembered that $\Pi$ is a
codimension $n$ linear space, and $dim(VSP(F,n+1)) = N-n$.
\end{proof}

For conics a bit improvement is at hand.

\begin{theorem}\label{th:con}
Let $F\in k[x_0,x_1, x_2]_2$ be a
general homogeneous polynomial of degree
two. Then  $\Vh{4}$ is birational to the
Grassmannian $\mathbb{G}(1,4)$, and
hence rational.
\end{theorem}

\begin{proof} The map is quite simple.
The $3$-planes passing through $[F] \in
\mathbb{P}^{5}$ are parametrized by
$\G(1,4)$ and a general linear space cuts
exactly 4 points on the Veronese surface
$V_{2,2}\subset\P^5$. To conclude it is
enough to check that $\dim\Vh{4}=\dim\G(1,4)=6$.
\end{proof}

We are not able to prove rationality for
arbitrary $n$ and $h$. Nonetheless the proof of
Theorem \ref{th:2rat} allow us to prove
the following unirationality statement.



\begin{theorem}
  Let $F\in k[x_0,\ldots,x_n]_2$ be a general homogeneous
polynomial of degree two. Then $VSP(F,h)$ is
unirational. 
\end{theorem}
\begin{proof}
We have to prove the statement for
$h>n+1$.
Let $\Pi\subset\P^N$ be a codimension
$n$ linear space and $q\in\Pi$ a point. 
The proof of Theorem \ref{th:rq} shows
that for a general $[F]\in\P^N$ there is
a well defined decomposition associated
to $q$. This can be seen as a rational
section
$$\sigma_q:\P^N\rat Sec_{n}(V_{2,n})$$
We proved that the general fiber
of the map
$\pi_{n}:Sec_{n}(V_{2,n})\to\P^N$ is
rational. Hence we have a well defined
birational map
$$\chi:\P^N\times\P^{N-n}\rat Sec_{n}(V_{2,n}).$$
This means that given a general
quadratic polynomial, say $q$, and a point in
$\P^{N-n}$ it is well defined an additive
decomposition of $q$ into $h$ factors. 
This allows us to
define the following map, for $h>n+1$
$$\psi_h:\P^{N-n}\times
(V_{2,n}\times\P^1)^{h-(n+1)}\rat \Vg{h} $$
given by 
\begin{eqnarray*}(p,[L_1^2],\lambda_1,\ldots,[L_{h-(n+1)}^2],\lambda_{h-(n+1)})\mapsto
  (\lambda_1L_1^2+\ldots+\lambda_{h-(n+1)}L^2_{h-(n+1)}+\\+\chi([F-\sum_{i=1}^{h-(n+1)}\lambda_1L_i^2],p)).
\end{eqnarray*}
The map $\psi_h$ is clearly generically
finite, of degree $\bin{h}{n+1}$, and
dominant. This is enough to
show that $\Vg{h}$ is unirational for
$h> n+1$.
\end{proof}

\section{Rational Connectedness}

In this section we prove the result on
rational connectedness taking advantage of the
preparatory work of the previous
sections.
In higher degree one cannot expect a
result like in the case of quadratic
polynomials . It is enough to think of
either Mukai Theorem, \cite{Mu}, where is
proven that $\Vh{10}$ is a $K3$ surface
for $F\in k[x_0,x_1,x_2]_6$ general, or Iliev and
Ranestad example of a symplectic $VSP$, \cite{IR1}.
On the other hand we found a nice
behaviour for infinitely many  degrees
and number of variables. Keep in mind
that $\Vh{h}$ are not empty only for $h\geq \frac{\bin{n+d}{n}}{n+1}$.

\begin{theorem} Assume that for some
  positive integer $k<n$ the number
  $\frac{\bin{d+n}{n}-1}{k+1}$ is an
  integer. Then the irreducible
  components of $\Vh{h}$ are Rationally
  Connected for $F\in k[x_0,\ldots,
    x_n]_d$ general and $h\geq \frac{\bin{n+d}{n}-1}{k+1}$.
\label{th:RC0}
\end{theorem}

Let us start stating explicitly
\cite[Remark 4.6]{Me1}. 
\begin{proposition} Let
  $V_{\delta,n}\subset\P^{N}$  be a
  Veronese embedding, for $\delta\geq 4$. Assume that
  $\cod \Sec_{h}(V)\geq n+1$. Then through
  a general point of $\Sec_h(V)$ there is a
  unique $h$-linear space $(h+1)$-secant
  to $V$.
\label{pro:unique}
\end{proposition}
\begin{proof}Let $z\in\Sec_h(V)$ be a
  general point. Assume that 
  $\langle p_0,\ldots,p_h\rangle\ni z$
  and $z\in\langle
  q_0,\ldots,q_h\rangle$ for $h$-tuple
  of points in $V$. Then Terracini Lemma, \cite{CC},
  yields 
$$\T_z\Sec_h(V)=\langle
  \T_{q_0}V,\ldots,\T_{q_h}V\rangle=\langle
  \T_{p_0}V,\ldots,\T_{p_h}V\rangle.$$
Therefore the general hyperplane section
$H\cap V$ singular at $\{p_0,\ldots,
p_h\}$ is singular at
$\{q_0,\ldots,q_h\}$ as well.
On the other hand, by \cite[Corollary
    4.5]{Me1}, $V$ is not $h$-weakly
  defective. Then by \cite[Theorem
    1.4]{CC} the general
  hyperplane section $H\cap V$ tangent at $h$-general
  points $\{p_0,\ldots,p_h\}$, of $V$ is
  singular only at those
  points. This gives $\{p_0,\ldots,p_h\}=
  \{q_0,\ldots,q_h\}$ and proves the proposition.
\end{proof}
\begin{proof}[Proof of Theorem
    \ref{th:RC0}] Without loss of
 generality, to simplify notation, we may assume that
    $\Vg{h}$ is irreducible. Fix
  $h=\frac{\bin{n+d}{n}-1}{k+1}=\frac{N}{k+1}$, and 
assume that $[\Lambda_x],[\Lambda_y]\in
\Vg{h}$ are two general points, with $\Lambda_x=\langle
x_1,\ldots x_h\rangle$ and $\Lambda_y=\langle
y_1,\ldots y_h\rangle$.

In the notation of
Proposition \ref{pro:W}, let $W_1:=W^n_{h,2h,n-k}$.
Let $[X]\in W_1$ be a general element,
then the numerical assumption, together
with \cite{AH} main Theorem, yields 
$$\dim \Sec_{h-1}(X)=h(k+1)-1=N-1.$$
Then $\Sec_{h-1}(X)\subset\P^N$ is an
hypersurface of degree say $\alpha$. 

Remark \ref{rem:relsec} allow us to define
a rational map as follows
$$\f:W_1\rat
\P(k[x_0,\ldots,x_N]_\alpha)$$ 
defined sending $X$ to its $h$-secant. Let
$SW_1=\overline{\f(W_1)}$ and
$SW_{1[F]}\subset SW_1$ be an irreducible
component of maximal dimension of the subvariety
parametrizing elements of $SW_1$
passing through a point $[F]\in\P^N$.
The variety $X$ is rational. Hence it is the projection
of a $\delta$-Veronese
embedding of $\P^k$, for some $\delta\gg
0$. Without loss of generality we may
assume that $\Sec_{h-1}(X)$ is the
projection of the $(h-1)$-secant variety of
the Veronese embedding (otherwise we
restrict to this irreducible component). Then by
Proposition \ref{pro:unique} there is a unique $h$-secant
linear space to $X$ through a general
point of $\Sec_{h-1}(X)$. 

We may then  define a rational dominant
map 
\begin{equation}\label{eq:dom}
\psi:SW_{1[F]}\rat\Vg{h}\subset\G(h-1,N)
\end{equation}
sending a general secant in $SW_{1[F]}$ to the unique
$h$-secant linear space passing through
$[F]\in\P^N$.  In the notation of
Proposition \ref{pro:W} we have
$$\overline{\psi^{-1}([\Lambda_x])}\supseteq
\f(W_{\{x_1,\ldots x_h\},n-k}),$$ 
$$\overline{\psi^{-1}([\Lambda_y])}\supseteq
\f(W_{\{y_1,\ldots,y_h\},n-k}),$$
and 
$$\overline{\psi^{-1}([\Lambda_x])}\cap \overline{\psi^{-1}([\Lambda_y])}\supseteq
\f(W_{\{x_1,\ldots,x_h,y_1,\ldots,y_h\},n-k}).$$

The subvarieties $W_{\{x_1,\ldots x_h\},n-k}$ and $W_{\{y_1,\ldots,y_h\},n-k}$ are
Rationally Connected. Therefore $SW_{1[F]}$ is
rationally chain connected by two
rational curves intersecting in a
general point of $\f(W_{\{x_1,\ldots,x_h,y_1,\ldots,y_h\},n-k})$. 

\begin{claim} The variety $ SW_{1[F]}$
  is Rationally Connected.
\label{cl:smooth}
\end{claim}
\begin{proof}[Proof of the Claim] The
  variety
  $SW_1\subset\P(k[x_0,\ldots,x_N]_\alpha)$
  parametrizes divisors in
  $\P^N$. Let
  $H_{[F]}\subset\P(k[x_0,\ldots,x_N]_\alpha)$
  be the hyperplane parametrizing the
  hypersurfaces passing through
  $[F]$. Then we have
  $SW_{1[F]}=SW_1\cap H_{[F]}$. 

A general point $[T]\in SW_1$ represents
a projection of the secant variety to a
Veronese. This yields that $T$  is singular in
codimension 1. Moreover a general point
of $\Sing(T)$ is a double point. That
is, by Proposition \ref{pro:unique},
for  $t\in \Sing(T)$  general
point there are two linear spaces
$h$-secant to the Veronese passing through
$t$. We can therefore assume that the general
point $x\in \f(W_{\{x_1,\ldots,x_h,y_1,\ldots,y_h\},n-k})$ is a general point of
$SW_1$.

Let $\Sigma_{[F]}\subset SW_{1[F]}$ be the
subvariety parametrizing secant
varieties with more than one
$(h-1)$-linear space $h$-secant passing
through $[F]$. First we compute the
codimension of $\Sigma_{[F]}$. We
already observed that for $[T]\in SW_1$
the hypersurface $T$ is singular along a
codimension $1$ set. Therefore the set of
hypersurfaces singular at a general
point $[F]\in \P^N$ is in codimension 2
in $SW_1$,
$$\cod_{SW_{1}}\Sigma_{[F]}=2.$$
All these hypersurfaces are clearly
contained in $SW_{1[F]}$, therefore we
conclude that
$$\cod_{SW_{1[F]}}\Sigma_{[F]}=1.$$
The above construction shows that $
SW_{1[F]}$ is rationally chain connected
by chains of rational curves passing
through general points of
$\Sigma_{[F]}$.
If the general point of $\Sigma_{[F]}$
is smooth in $ SW_{1[F]}$ then the claim
is proved. 

Assume that $ SW_{1[F]}$ is
singular along $\Sigma_{[F]}$. Let
$\nu:Z\to  SW_{1[F]}$ be the
normalization, and
$$S_{\{x_i\}\{y_j\}}:=\f(W_{\{x_1,\ldots,x_h\},n-k})\cap
\f(W_{\{y_1,\ldots,y_h\},n-k})$$ be the intersection.
By construction we have the following:
\begin{itemize}
\item[(a)]$\dim S_{\{x_i\}\{y_j\}}\geq h$;
\item[(b)] for a general $s\in\Sigma_{[F]}$
  there are exactly two $h$-tuples  such
  that $S_{\{x^s_i\}\{y_j^s\}}\ni s$.
\end{itemize}
Let us consider $\Sigma_{[F]}$ with its complex
topology.
The morphism $\nu$ is a finite covering outside
a codimension 1 set, say $K$. For any
point $s\in K^c$ there
is an open neighborhood (in the complex
topology), say $B_s$, such that
$\nu_{|\nu^{-1}(B_s)}$ is finite and
$\rm\acute{e}$tale. The set $K$  is closed and
of measure zero. That is for any
$\epsilon>0$ there is an open $V\subset
\Sigma_{[F]}$ such that $V\supset K$
and $V$ has measure bounded by
$\epsilon$.
Then $V^c$ is a compact space and we may
cover it with finitely many open sets
$\{B_{s_i}\}_{i=1,\ldots,m}$. Then, by
(a) and
(b) above, we may assume that for two
general decompositions
$$\dim(\f(W_{\{x_1,\ldots,x_h\},n-k})\cap
\f(W_{\{y_1,\ldots,y_h\},n-k})\cap
V^c)>0.$$
By construction
$\nu_{|\nu^{-1}(B_s)}$ is finite and
$\rm\acute{e}$tale. The compact $V^c$ is covered by a finite
number of $B_s$. 
This, going back to Zariski topology and
keeping in mind (a),
together with the rational connectedness of the
subvarieties $W$, shows that the $MRC$
fibration of $Z$ has at most finitely many
fibers. Hence the irreducible varieties 
$Z$ and $SW_{1[F]}$ are
Rationally  Connected.
\end{proof}

The claim shows that $SW_{1[F]}$  is
Rationally Connected and
hence $\Vg{h}$ is
Rationally Connected, via the map $\psi$
of equation (\ref{eq:dom}). To conclude the
proof for $h>\frac{\bin{n+d}{n}}{k+1}$ it is
then enough to apply Theorem \ref{th:chain}.
\end{proof}

For special values a more precise
statement con be obtained.

\begin{theorem}\label{th:mix}
The variety $\Vh{h}$ is Rationally
Connected in the following cases:
\begin{itemize}
\item[a)] $F\in k[x_0,x_1,x_2]_4$ and
  $h\geq 6$,
\item[c)] $F\in
k[x_0,\ldots,x_4]_3$ and $h\geq 8$,
\item[b)] $F\in
k[x_0,\ldots,x_3]_3$ and $h\geq 6$,
\item[d)] $F\in k[x_0,x_1,x_2]_3$ and
  $h\geq 4$,
\end{itemize}
The variety $\Vh{h}$ is uniruled for  $F\in
k[x_0,\ldots,x_4]_3$ and $h\geq 7$.
\end{theorem}
\begin{proof}
In cases a) and b) we know that 
$\Vh{6}$, \cite{Mu}, and $\Vh{8}$, \cite{RS}, respectively are
rational of dimension $n+1$. Then to
conclude it is enough to apply
Theorem \ref{th:chain}.

In case c) observe that there is a twisted cubic in
$\P^3$ through 6 points. Then Theorems
\ref{th:1} and \ref{th:RC}
produce a chain of $\P^2$
through very general points of
$VSP(F,6)$. Then we apply 
Theorem \ref{th:chain} to conclude for
arbitrary $h\geq 7$.
In case d) we have $\P^2\iso\Vh{4}$
and  we
conclude again by Theorem \ref{th:chain}.

Finally observe that there is a rational
quartic in $\P^4$ through 7 points. Then Theorems \ref{th:1} and \ref{th:RC}
produce  a
$\P^1$ through a general point of
$VSP(F,h)$, for $h\geq 7$.
\end{proof}


\begin{thebibliography}{99999}
\bibitem[AH]{AH} J. Alexander, A. Hirschowitz,  \textit{Polynomial
  interpolation in several variables},  J. Algebraic Geom.  {\bf 4}
  (1995),  no. 2, 201--222
\bibitem[CC]{CC} L. Chiantini, C. Ciliberto \textit{ Weakly defective
  varieties} Trans. Amer. Math. Soc. {\bf 354} (2002), no. 1, 151--178.
\bibitem[Di]{Di} L. Dickson,
  \textit{History of the theory of
    numbers. Vol. II: Diophantine
    analysis}, Chelsea Publishing Co.,
  New York 1966 xxv+803 pp.
\bibitem[Dx]{Dx} A. C. Dixon,
  \textit{The canonical forms of the
    ternary sextic and quaternary
    quartic}, 
London Math. Soc. Proc. (2) {\bf 4}
(1906). 
\bibitem[Do]{Do} I. Dolgachev, \textit{Dual homogeneous forms and varieties of power sums}, Milan Journal of Mathematics, 99.
\bibitem[DK]{DK} I. Dolgachev, V. Kanev, \textit{Polar covariants of plane cubics and quartics}, Adv.
in Math. 98 (1993), 216301.
\bibitem[GHS]{GHS} T. Graber, J. Harris, J. Starr, \textit{Families of rationally connected
varieties}, Preprint, 2001. 
\bibitem[Hi]{Hi} D. Hilbert, \textit{Letter adresse\'e \`a
  M. Hermite}, Gesam. Abh. {\bf vol II} 148-153.
\bibitem[IK]{IK} A. Iarrobino, V. Kanev,
  \textit{Power Sums, Gorenstein
    Algebras and Determinantal
    Loci}, Lecture notes in Mathematics,
  1721, Springer, 1999.
\bibitem[IR1]{IR1} A. Iliev, K. Ranestad, \textit{$K3$ surfaces of
  genus 8 and varieties of sums of powers of cubic fourfolds},
  Trans. Amer. Math. Soc.  {\bf 353}  (2001),  no. 4, 1455--1468.
\bibitem[IR2]{IR2} \bysame \textit{Canonical curves and
  varieties of sums of powers of cubic
  polynomials}, J. Algebra  {\bf 246}
  (2001),  no. 1, 385--393.
\bibitem[Ko]{Ko} J. Koll\'ar,
  \textit{Rational Curves on Algebraic Varieties}, Ergebnisse der Math, {\bf 32},
(1996), Springer.
\bibitem[KMM]{KMM} J. Koll\'ar, Y. Miyaoka,
  S. Mori, \textit{Rationally connected
    varieties}, J. Alg. Geom. 1 (1992),
  429-448.
\bibitem[Me1]{Me1} M.Mella, \textit{Singularities of linear systems and
  the Waring problem}, Trans. Amer. Math. Soc.  358  (2006),  no. 12, 5523--5538.
\bibitem[Me2]{Me2} \bysame, \textit{Base
  Loci of linear systems and
  the Waring problem},
  Proc. Amer. Math. Soc. 137  (2009),
  no. 1, 91--98.
\bibitem[Mu1]{Mu} S. Mukai, \textit{Fano $3$-folds},
Complex projective geometry (Trieste, 1989/Bergen, 1989), 255--263,
London Math. Soc. Lecture Note Ser., 179,
Cambridge Univ. Press, Cambridge, 1992.
\bibitem[Mu2]{Mu2}\bysame \textit{Polarized K3 surfaces of genus 18 and 20}, In Complex Projective
Geometry, LMS Lecture Notes Series, Cambridge University Press, 1992,
264-276.
\bibitem[Pa]{Pa} F. Palatini, \textit{Sulla rappresentazione delle forme
  ternarie mediante la somma di potenze di forme lineari},
  Rom. Acc. L. Rend. {\bf 12} (1903) 378-384.
\bibitem[RS]{RS} K. Ranestad,
  F.O. Schreier, \textit{Varieties of
    Sums of Powers}, J. Reine
  Angew. Math, 525, 2000.
\bibitem[Ri]{Ri} H.W. Richmond, \textit{On canonical forms}, Quart. J. Pure
  Appl. Math. {\bf 33} (1904) 967-984.
\bibitem[Sy]{Sy} J.J. Sylvester, \textit{Collected works}, Cambridge
  University Press (1904).
\bibitem[TZ]{TZ} I. Takagi, F. Zucconi,
  \textit{Scorza quartics of trigonal
    spin curves and their varieties of
    power sums}, to appear on
  Math. Ann. arXiv:0801.1760v1
  [math.AG].
\end{thebibliography}
\end{document}